\documentclass{amsart}
\usepackage{amsmath, amsfonts, amssymb}
\usepackage{mathrsfs}
\usepackage{hyperref} 
\usepackage{enumerate} 
\usepackage{tikz}

\usetikzlibrary{matrix}

\newtheorem{theorem}{Theorem}[section]
\newtheorem{lemma}[theorem]{Lemma}
  \newtheorem{corollary}[theorem]{Corollary}
\newtheorem{definition}[theorem]{Definition}
\newtheorem{example}[theorem]{Example}
 
\newtheorem{remark}[theorem]{Remark}
\newtheorem{remarks}[theorem]{Remarks}

\numberwithin{equation}{section}

\def\N{\mathbb{N}}

\newcommand{\classX}{{\mathfrak X}}
\newcommand{\classY}{{\mathfrak Y}}

\def\core{\mathrm{Core}}

\def\opa{\mbox{\scriptsize\sf A}}
\def\opec{\mbox{\scriptsize\sf C}}
\def\opd{\mbox{\scriptsize\sf D}}
\def\ope{\mbox{\scriptsize\sf E}}
\def\opl{\mbox{\scriptsize\sf L}}
\def\opn{\mbox{\scriptsize\sf N}}
\def\opp{\mbox{\scriptsize\sf P}}
\def\opq{\mbox{\scriptsize\sf Q}} 
\def\opr{\mbox{\scriptsize\sf R}} 
 
\def\ops{\mathsf{s}}

\begin{document}

\title{Classes of Algebras and closure operations}

\author{I. S. Gutierrez}
\address{Department of Mathematics and Statistics, Universidad del Norte, Km 5 via a Puerto Colombia, Barranquilla - Colombia}
\curraddr{Department of Mathematics and Statistics, Universidad del Norte, Km 5 via a Puerto Colombia, Barranquilla - Colombia}
\email{isgutier@uninorte.edu.co}

\author{Anselmo Torresblanca-Badillo}
\address{Department of Mathematics and Statistics, Universidad del Norte, Km 5 via a Puerto Colombia, Barranquilla - Colombia}
\curraddr{Department of Mathematics and Statistics, Universidad del Norte, Km 5 via a Puerto Colombia, Barranquilla - Colombia}
\email{atorresblanca@uninorte.edu.co}
\thanks{ }

\author{David A. Towers}
\address{Department of Mathematics and Statistics, Lancaster University, , Lancaster - UK}
\curraddr{Department of Mathematics and Statistics, Lancaster University, , Lancaster - UK}
\email{d.towers@lancaster.ac.uk}
\thanks{ }

\subjclass[2010]{Primary 68R05; Secondary 05C69}

\date{\today}

\dedicatory{In memory of Prof. Dr. Klaus Doerk}

\keywords{Non-associative algebras, finite-dimensional algebra, solvable algebras, nilpotent algebras, supersolvable algebras, classes of algebras, formations, Schunck classes,  closure operations}

\begin{abstract}
The calculus of classes and closure operations has proved to be a useful tool in group theory and has led to a deep theory in the study of finite soluble groups. More recently, parallel theories have started to be developed in various varieties of algebras, such as Lie, Leibniz and Malcev algebras. This paper seeks to investigate the extent to which these later theories can be generalised to the variety of all non-associative algebras. 
\end{abstract}

\maketitle

\section{Introduction}
 In the early 1960s Philip Hall introduced a systematic use of closure operations into group theory (see \cite{Hall1}, \cite{Hall2}). This was followed by the development of the concept of a formation by Gasch\"{u}tz in 1963. Since then, these concepts have contributed significantly to the theory of finite soluble groups, as is expounded in \cite{Doerk-Hawkes} and \cite{Ballester-Ezquerro}.  
\par

In recent years, closure operators have been more widely applied in various mathematical systems. In particular, a parallel theory has been developed in the context of a universal algebra (see \cite{Shemetkov}, \cite{Shemetkov-Skiba}, \cite{Skiba}, for instance), and in various varieties of algebras, such as Lie algebras (see \cite{Barnes2}, \cite{B-G-H}, \cite{Stit}, for example), Leibniz algebras (see \cite{Barnes} and \cite{Barnes-x}, for example) and Malcev algebras (\cite{Stit2}). It is the purpose of this paper to investigate the extent to which some of these last three sets of results can be generalised to the variety of all non-associative algebras. 

In section 2 we introduce the concepts of classes of algebras and  of closure operations and present their basic properties. In section 3 we introduce a partial order on closure operations, define the join of a set of closure operations and show that the join of various such operations are again closure operations. In the fourth section the concepts of homomorph, formation, saturation, Schunck class and Fitting class are presented, and some equivalent conditions for a non-empty homomorph and for a non-empty formation to be saturated are proved. The final section is devoted to conclusions and ideas for further work.

Throughout this work, unless otherwise stated, all of the algebras considered are assumed to be non-associative algebras over a field $K$.

 An algebra is said to be \texttt{simple} if it has only the trivial ideals 0 and $A$ and it is not the one-dimensional algebra. The \texttt{core} of a subalgebra $H$ of $A$ with respect to $A$, denoted by $\core_A(H)$, is defined to be the largest ideal of $A$ contained in $H$. An algebra  is called \texttt{primitive} if it has a core-free maximal subalgebra. 

If $S\subseteq A$, then the ideal of $A$ \texttt{generated} by  $S$ or the \texttt{ideal closure} of $S$ in $A$ is defined to be the unique smallest ideal of $A$ containing $S$. That is, the intersection of all ideals of $A$ which contain $S$, and it is denoted like in group theory by $S^A$.  It is straightforward to check that   
\begin{equation}\label{idealclosure}
	S^A = \bigg\{\sum_{s\in S}\bigg(\sum_{k=1}^{n_s} a_ksy_k\bigg) \mid n_s\in \N, \ x_k, y_k\in A  \bigg\}.
\end{equation}

A subspace $I$ is called a \texttt{subideal} of $A$  if there exist a chain of subalgebras
\begin{equation}
I = I_0 < I_1 < \cdots < I_n = A,
\end{equation}
where $I_j$ is an ideal of $I_{j+1}$ for each $0\leq j\leq n-1$.

The Frattini subalgebra of $A$, denoted by $F(A)$, is defined as the intersection of the maximal subalgebras of $A$; if $A$ has no maximal subalgebras then it is defined to be $A$. E. I. Marshall \cite{E. I. Marshall} proved that for finite-dimensional Lie algebras over an algebraically closed field of characteristic zero the Frattini subalgebra is an ideal and D. A. Towers \cite{Towers1} showed that the restriction of algebraic closure can be removed to obtain that it is a characteristic ideal. However, in general, the Frattini subalgebra need not be an ideal, so we define the Frattini ideal, $\Phi(A)$, to be equal to $\core_A(F(A))$.

Set $A^1 = A^{(1)} = A$, and then for $n\geq 1$ define inductively
\begin{equation}
A^{n+1} = \sum_{i+j=n+1} A^i A^j,
\end{equation}
and 
\begin{equation}
	A^{(n+1)} =A^{(n)}A^{(n)} = (A^{(n)})^2.
\end{equation}
The subset $A^n$ is called the $n$-th power of $A$. The chain of subsets
\begin{equation}
A^1 \supseteq A^2\supseteq \cdots\supseteq A^n \supseteq \cdots
\end{equation}
is a chain of ideals of $A$. A non-associative algebra $A$ is called \texttt{nilpotent} if $A^n = 0$ for some $n$ and  \texttt{solvable} if $A^{(r)} = 0$ for some $r$. The smallest natural number $n$ (respectively $r$) with this property is called the \texttt{nilpotency index}  (respectively \texttt{solvability index}) of $A$. Note that the elements of the chain
\begin{equation}
	A^{(1)} \supseteq A^{(2)} \supseteq \cdots\supseteq A^{(n)} \supseteq \cdots
\end{equation}
beyond $A^{(2)}$, are not, generally speaking, ideals in $A$.

Motivated by the concept of supersolubility in group theory a similar concept has been introduced in Lie algebras, where normal subgroups with cyclic factors
were replaced by ideals with one dimensional factors. In general, we say that an algebra $A$ is called \texttt{supersolvable} if it is solvable and has a flag of ideals $${0}=A_0\leq A_1\leq \ldots \leq A_n=A$$
where $\dim A_i=i$ for $0 \leq i \leq n$. For a general non-associative algebra $A$, the added condition that $A$ be solvable is needed, as the one-dimensional algebra spanned by a single idempotent has such a flag, trivially, but is not solvable.

 If $A$ is an algebra over a field $K$, an \texttt{$A$-bimodule} is a vector space $M$ over $K$ with bilinear mappings
$$ A\times M \rightarrow M, (x,m)\mapsto xm  \hbox{ and } M\times A \rightarrow M, (m,x)\mapsto mx.
$$
The \texttt{split null extension} of $A$ by $M$, $A\rtimes M$, is the algebra over $K$ with underlying vector space $A\oplus M$ and multiplication
$$ (x+m)(y+n)=xy+(xn+my) \hbox{ for all } x,y\in A, m,n\in M. $$
The \texttt{regular bimodule}, $Reg(A)$, is the underlying vector space of $A$ considered as a bimodule with $mx$ and $xm$ as the multiplication in $A$.

\section{Classes of Algebras and Closure Operations}

A \texttt{class} of algebras $\mathfrak{X}$ is a collection of algebras with the property that if $A\in \mathfrak{X}$, then every algebra isomorphic to $A$ belongs to $\mathfrak{X}$. The algebras which belong to a class $\mathfrak{X}$ are called $\mathfrak{X}$-algebras.

Following K. Doerk and T. O. Hawkes \cite{Doerk-Hawkes} and A. Ballester-Bolinches and L. Ezquerro \cite{Ballester-Ezquerro}, we denote the empty class of algebras by $\emptyset$ and we use Gothic font when a single capital letter denotes a class of algebras.

The following are examples of specific classes  of algebras:
\begin{align*}
\emptyset &  \ \ \ \text{denotes the empty class of algebras} \\
\mathfrak{A} &  \ \ \ \text{denotes the class of all abelian algebras} \\
\mathfrak{N} &  \ \ \ \text{denotes the class of all nilpotent algebras} \\
\mathfrak{N}_c &  \ \ \ \text{denotes the class of all nilpotent algebras of index $\leq $ c, \ $c\in \N$} \\
\mathfrak{S} & \ \ \ \text{denotes the class of all solvable algebras} \\
\ddot{\mathfrak U} &  \ \ \ \text{denotes the class of all supersolvable algebras} \\
\mathfrak{E} & \ \ \ \text{denotes the class of all finite dimensional algebras}.
\end{align*}

Algebra classes are partially ordered by inclusion and we write $\mathfrak{X} \subseteq \mathfrak{Y}$ to denote that the class $\mathfrak{X}$ is a subclass of the class $\mathfrak{Y}$.

A map which sends a class of algebras to a class of algebras is called a class map. Among class maps are the so-called \texttt{closure operations}, which play a central role in studying of algebra classes. 

A class map $\opec$ assigning to each class of algebras $\mathfrak{X}$ a class of algebras $\opec\mathfrak{X}$ is called an \texttt{operation}, if it holds the following conditions:
\begin{enumerate}
\item[(C1)] $\opec \emptyset = \emptyset$ \ and
\item[(C2)] $\mathfrak{X} \subseteq \opec \mathfrak{X} \subseteq \opec \mathfrak{Y}$, providing $\mathfrak{X} \subseteq \mathfrak{Y}$. (We say $\opec$ is expanding and monotonic respectively).
\end{enumerate}
An operation $\opec$ is called a \texttt{closure operation}, if it is idempotent, that is, if
\begin{enumerate}
\item[(C3)]  $\opec \mathfrak{X} = \opec(\opec \mathfrak{X})$, for all class $\mathfrak{X}$. 
\end{enumerate}

\begin{definition}\label{1.5}
Let $\mathfrak{X}$ be a class of algebras. Then the following are important examples of class maps:
\begin{align*}
\ops \mathfrak{X} & = (A\mid A\leq K, \ \text{for some} \ K \in \mathfrak{X})\\
\ops_n \mathfrak{X} & = (A\mid A \mbox{ a subideal of } K \ \text{for some} \ K\in \mathfrak{X})\\
\bar{\ops}_n \mathfrak{X} & = (A\mid K \mbox{ a subideal of } A, \ \text{for some} \ K\in \mathfrak{X})\\
\opd_0 \mathfrak{X} & =  (A\mid A= \bigoplus\nolimits_{j=1}^r I_j \ \mbox{with each}\ I_j\in \mathfrak{X})\\ 
\opq \mathfrak{X} & =  (A\mid \exists \, H \in \mathfrak{X} \ \text{and} \ \exists \ \varphi : H \longrightarrow A \ \text{an epimorphism})\\ 
\opr_0 \mathfrak{X} & = (A\mid \exists\ I_1, \ldots, I_r\unlhd A\ \mbox{with}\ A/I_j \in \mathfrak	X\ \mbox{and}\ \bigcap\nolimits_{j=1}^r I_j = 0)\\
\opr\mathfrak{X} & = (A\mid \exists \ (I_j)_{j\in J},\ I_j\unlhd A\ \mbox{with}\ A/I_j \in \mathfrak	X\ \mbox{for all $j\in J$ and}\ \bigcap\nolimits_{j\in J} I_j = 0)\\
\opn_0 \mathfrak{X} & =  (A\mid \exists \, K_1, \ldots, K_r	\mbox{ subideals of } A\ \mbox{with}\ K_i \in \mathfrak{X}\ \mbox{and}\ A = \sum\nolimits_{i=1}^r K_i )\\ 
\opn \mathfrak{X} & =  (A\mid \exists (I_j)_{j\in J}, I_j \mbox{ a subideal of $A$}\ \mbox{with}  \ I_j \in \mathfrak{X} \ \mbox{for all $j\in J$ and } A = \sum\nolimits_{j\in J} I_j )\\ 
\ope_\Phi \mathfrak{X}  & =  (A\mid \exists\ N \unlhd A\	\mbox{with}\ N \leq \Phi(A)\ \mbox{and}\ A/N \in \mathfrak{X})\\
\opp \mathfrak{X} &= (A \mid \mbox{\scriptsize\sf Q}(A) \cap \mathfrak P \subseteq \mathfrak{X}).
\end{align*}
\end{definition}

Note that, if $\mathfrak{X}$ is a class of algebras, then $\ops \mathfrak{X}$ is the class of all subalgebras of $\mathfrak{X}$-algebras, and $\opq \mathfrak{X} $ is the class of all quotients of $\mathfrak{X}$-algebras. 

A class $\mathfrak{X}$ is called $\opec$-closed if $\mathfrak{X} = \opec\mathfrak{X}$. If $\opec$ is a closure operation, then is $\opec\mathfrak{X}$ a $\opec$-closed class, for all classes $\mathfrak{X}$. It is clear from (C1) and (C2) that the classes $\emptyset$ and $\mathfrak{E}$ are $\opec$-closed for any operation $\opec$. 

In this terminology $\ops \mathfrak{X} = \mathfrak{X}$ means that $\mathfrak{X}$ is closed under subalgebras, $\opq \mathfrak{X} = \mathfrak{X}$ indicate that $\mathfrak{X}$ is quotient-closed, whereas $\opd_0 \mathfrak{X} = \mathfrak{X}$ implies that $\mathfrak{X}$ is closed under forming direct sums.

\begin{definition}
If $\mathfrak{X}$ and $\mathfrak{Y}$ are two classes of algebras, we define their product $\mathfrak{X} \mathfrak{Y}$ as 	follows: 
\begin{equation}\label{class-product}
\mathfrak{XY} = (A\mid \exists \ N\unlhd A, \ N\in \mathfrak{X} \ \mbox{with} \ A/N \in \mathfrak{Y})
\end{equation}
Algebras in the class $\mathfrak{XY}$ are called $\mathfrak{X}$-by-$\mathfrak{Y}$ algebras.
If $\mathfrak{X} = \emptyset$ or $\mathfrak{Y} = \emptyset$, then we define $\mathfrak{XY} = \emptyset$.

For powers of a class, we set $\mathfrak{X}^0 = (1)$, and for $n\in\N$ make the inductive definition $\mathfrak{X}^n = (\mathfrak{X}^{n-1})\mathfrak{X}$. An algebra in $\mathfrak{X}^2$ is called \texttt{meta}-$\mathfrak{X}$. 
\end{definition}
 
\begin{remark}
This product is neither associative nor commutative in general.

We will find conditions on $\mathfrak{X}$ and $\mathfrak{Y}$ such that $\mathfrak{X(YZ)} =\mathfrak{(XY)Z}$ does hold.
\end{remark}

\begin{lemma}
With the exception of $\opp$, all the class maps  in definition \ref{1.5} are closure operations. 
\end{lemma}

\begin{proof}
It is very easy to verify that 
all the class maps are expanding and monotonic. We show now that all this class maps are idempotent. Notice that, if $\opec$ is an expanding class map, then $\opec\mathfrak{X}\subseteq \opec(\opec\mathfrak{X})$, for all class $\mathfrak{X}$. Therefore it is sufficient to demonstrate that $\opec(\opec\mathfrak{X}) \subseteq \opec\mathfrak{X}$, with $\opec$ the above defined class maps.
\begin{enumerate}[(1)]
\item  Operator $\ops$. Let $A\in \ops(\ops \mathfrak{X})$. Then there exists $H\in \ops \mathfrak{X}$ such that $A\leq H$. Similarly, there is an algebra $V\in \mathfrak{X}$ with $H\leq V$. Then we have $A\leq H \leq V$, with $V\in \mathfrak{X}$; that is $A\in \ops \mathfrak{X}$.

\item  Operator $\ops_n$. Let $A\in \ops_n(\ops_n \mathfrak{X})$. Then there exist $H\in \ops_n \mathfrak{X}$ and $K\in \mathfrak{X}$ such that $A$ is a subideal of $H$ and $H$ is a subideal of $K$. This proves that $A$ is a subideal of  $K$ with $K\in \mathfrak{X}$; that is, $A\in \ops_n \mathfrak{X}$.

\item Operator $\bar{\ops}_n$. Let $A\in \bar{\ops}_n(\bar{\ops}_n\mathfrak{X})$. Then there exist $H\in \bar{\ops}_n\mathfrak{X}$ and $V\in \mathfrak{X}$  such that $V$ is a subideal of $H$ and $H$ is a subideal of $A$. This proves that $A\in \bar{\ops}_n\mathfrak{X}$. 

\item  Operator $\opd_0$. Let $A\in \opd_0(\opd_0 \mathfrak{X})$. Then $A= \bigoplus_{i=1}^r K_i$ with $K_i\in \opd_0\mathfrak{X}$, $i=1,\ldots, r$. Therefore 
each algebra $K_i$ has the form $K_i = \bigoplus_{j=1}^{r_i} N_{ij}$ with $N_{ij}\in \mathfrak{X}$. Since $\bigoplus_{i=1}^{r} \bigoplus_{j=1}^{r_i} N_{ij} = \bigoplus_{i=1}^r N_i = A$, it follows that $A\in \opd_0 \mathfrak{X}$.

\item  Operator $\opq$. Let $A\in \opq(\opq \mathfrak{X})$. Then there exist algebras $H\in \opq \mathfrak{X}$ and $V\in \mathfrak{X}$ and epimorphisms $\varphi : H\longrightarrow A$ and $\psi : V\longrightarrow H$. Then we have an epimorphism $\varphi\psi : V\longrightarrow A$, with $V\in \mathfrak{X}$; that is $A\in \opq \mathfrak{X}$.

\item Operator $\opr_0$.  Let $A\in \opr_0(\opr_0 \mathfrak{X})$. Then $A$ has ideals $N_1,\ldots, N_r$ such that $A/N_i\in \opr_0 \mathfrak{X}$ and $\bigcap_{i=1}^r N_i = 0$. Therefore each algebra $A/N_i$ has ideals $K_{ij}/N_i$, $j=1,\ldots, r_i,$ such that $A/K_{ij} \cong (A/N_i)/(K_{ij}/N_i)$ and $\bigcap_{j=1}^{r_i} K_{ij} = N_i$. Since $\bigcap_{i=1}^{r}\bigcap_{j=1}^{r_i} K_{ij} = \bigcap_{i=1}^r N_i = 0$, it follows that $A\in \opr_0 \mathfrak{X}$.

\item Operator $\opr$. Is similar to (6).

\item Operator $\opn_0$. Let $A\in \opn_0(\opn_0 \mathfrak{X})$. Then $A$ has subideals $K_1,\ldots, K_r$ such that $K_i\in \opn_0 \mathfrak{X}$ and
$A = \sum_{i=1}^r K_i$. Therefore each subalgebra $K_i$ has subideals $N_{ij}$, $j=1,\ldots, r_i$ with $N_{ij}\in  \mathfrak{X}$ and $K_i = \sum_{j=1}^{r_i} N_{ij}$. Since $\sum_{i=1}^{r} \sum_{j=1}^{r_i} N_{ij} = \sum_{i=1}^r N_i = A$, it follows that $A\in \opn_0 \mathfrak{X}$.

\item Operator $\opn$. Is similar to (8).

\item  Operator $\ope_{\Phi}$.  Let $A\in \ope_{\Phi}(\ope_{\Phi} \mathfrak{X})$. Then $A$ has an ideal $N\leq \Phi(A)$ such that $A/N\in \ope_{\Phi} \mathfrak{X}$. Consequently $A/N$ possesses an ideal $K/N\leq \Phi(A/N)$ such that $A/K \cong (A/N)/(K/N) \in \mathfrak{X}$. Since $\Phi(A/N) = \Phi(A)/N$, we have $K\leq \Phi(A)$. Therefore $A\in \ope_{\Phi}\mathfrak{X}$.
\end{enumerate}
\end{proof}

\begin{remark}
Let $\opec$ be a closure operation. If $\{\mathfrak{X}_j\}_{j\in J}$ is a family of $\opec$-closed algebras classes, it follows immediately that the class $\bigcap_{j\in J} \mathfrak{X}_j$ is also $\opec$-closed.
\end{remark}

A product of operations is defined as follows:
	
\begin{definition}
Let $\opec_1$ and $\opec_2$ be operations. The product $\opec_1 \opec_2$  is defined by composition thus $(\opec_1 \opec_2)\mathfrak{X} = \opec_1 (\opec_2 \mathfrak{X})$, for all classes $\mathfrak{X}$.
\end{definition}

\begin{remark}
If $\opec_1$ and $\opec_2$ are closure operations, then $\opec_1 \opec_2$  need not be, in general, a closure operation. For example, if we define the class maps $\ope$, and $\opl$ as follows:
\begin{align*}
\ope \mathfrak{X}  &= (A\mid \exists\  0=I_0 \unlhd I_1\unlhd \cdots \unlhd I_n\unlhd A\ \mbox{with}\ I_{j+1}/I_j \in \mathfrak	X \ \mbox{for all}\ j=0,\ldots, n-1)\\
\opl\mathfrak{X}  &= (A\mid \forall \,  S\subset A, \ S \ \text{finite} \ \exists \, H\leq A, \ \text{with} \ H\in \mathfrak	X\ \mbox{and}\ S\subseteq H)
\end{align*}
It is easy to show that they are closure operations, and in context of infinite-dimensional Lie algebras, $\ope \opl \mathfrak{A} = \ope\mathfrak{A}$, whereas $\ope \opl \ope \opl \mathfrak{A} = \ope \opl \ope\mathfrak{A}$ contains non-solvable algebras, for instance the direct sum of a set of solvable algebras of unbounded derived length. 
\end{remark}

\begin{remark}
Let $\opec$ be a closure operation. Conditions (C2) and (C3) imply that  $\opec \mathfrak{X} $ is the uniquely determined, smallest $\opec$-closed class that contains the class $\mathfrak{X}$. 
\end{remark}

\begin{remark}
Let $\opec$ be a closure operation. If $\{\mathfrak{X}_j\}_{j\in J}$ is a family of $\opec$-closed algebras classes it follows immediately that the class $\bigcap_{j\in J} \mathfrak{X}_j$ is also $\opec$-closed.
\end{remark}

Some elementary closure properties of class product are discussed now.

\begin{lemma}\label{Direct-product}
Given algebras $A_1, A_2, H$ and epimorphisms $\mu_i : A_i\longrightarrow H$, let be $N_i := \ker(\mu_i)$. Then 
\begin{equation}\label{PataGallina}
A_1 \curlywedge A_2 := \{(a_1,a_2) \mid a_i\in A_i, \ \mu_1(a_1) = \mu_2(a_2)\}
\end{equation}
is a subalgebra of $A_1\oplus A_2$. Furthermore, there exists epimorphisms $\alpha_i : A_1 \curlywedge A_2 \longrightarrow A_i$ ($i=1, 2$) with 
\[\ker(\alpha_1) = \{(0,a_2) \mid a_2\in N_2\} \cong N_2\]
	and 
\[\ker(\alpha_2) = \{(a_1,0) \mid a_1\in N_1\}\cong N_1.\]
In addition there is an epimorphism $\mu : A_1 \curlywedge A_2\longrightarrow H$ with
\[\ker(\mu) = \ker(\alpha_1) \oplus \ker(\alpha_2). \]
\end{lemma}

\begin{proof}
It is clear that $A_1 \curlywedge A_2$ is a subalgebra of $A_1\oplus A_2$. Let $\pi_i : A_1\oplus A_2 \longrightarrow A_i$ the homomorphism defined by 
	\[\pi_i(a_1,a_2):= a_i.\]
Then the restriction of $\pi_i$ on $A_1 \curlywedge A_2$ is an epimorphism $\alpha_i$ with
\begin{align*}
\ker(\alpha_1) & = \{(a_1,a_2) \mid a_i\in A_i, \ \mu_1(a_1) = \mu_2(a_2), \ a_1=0  \}\\
& = \{(0,a_2) \mid a_i\in A_i, \ 0=\mu_1(a_1) = \mu_2(a_2)  \}\\
& = \{(0,a_2) \mid a_2\in \ker(\mu_2)  \}\\
&= 0\oplus N_2.
\end{align*}
On the other hand $\mu(a_1,a_2):= \mu_1(a_1) = \mu_2(a_2)$ defines an epimorphism $\mu$ from   $A_1 \curlywedge A_2$ to $H$. Evidently $\ker(\mu) = \ker(\alpha_1) \oplus \ker(\alpha_2)$.
\end{proof}

\begin{lemma}
If $\mathfrak{X}$ and $\mathfrak{Y}$ are classes of algebras, then $\mathfrak{Y}(\opq \mathfrak{X}) \subseteq \opq (\mathfrak{Y}\mathfrak{X})$. 
\end{lemma}

\begin{proof}
Let $A\in \mathfrak{Y}(\opq \mathfrak{X})$. Then $A$ has an ideal $I\in \mathfrak{Y}$ with $A/I\in \opq \mathfrak{X}$. Thus there exists an algebra $H\in \classX$ and an epimorphism $\mu_1 : H \longrightarrow A/I$. Let $\mu_2 : A\longrightarrow A/I$ be the canonical epimorphism.   Consider now the subalgebra 
\begin{equation}\label{PataGallina2}
S := H \curlywedge A := \{(h,l) \mid h\in H, \ l\in A, \ \mu_1(h) = \mu_2(l) = l+I\}
\end{equation}
of the direct sum $H\oplus A$. Using Lemma \ref{Direct-product}  it follows that there exist epimorphisms $\alpha_1 : S  \longrightarrow H$ and $\alpha_2 : S \longrightarrow A$ with 
\begin{align*}
\ker(\alpha_1) & = \{(0,i) \mid i\in I\} = 0\oplus I \cong I\in \classY, \ \text{and}\\
\ker(\alpha_2) & = \{(h,0) \mid h\in \ker(\mu_1)\} = \ker(\mu_1) \oplus 0 \cong \ker(\mu_1).
\end{align*}
Then $S/\ker(\alpha_1) \cong H\in \classX$ and we have that $S\in \classY\classX$. 

On the other hand $S/\ker(\alpha_2) \cong A$, and therefore $A\in \opq (\mathfrak{Y}\mathfrak{X})$.
\end{proof}

\begin{lemma}
Let $\opec$ be any of the closure operations $\ops, \opq$ or $\ops_n$. If $\classX$ and $\classY$ are $\opec$-closed classes of algebras, then $\mathfrak{XY}$ is also a $\opec$-closed class. 
\end{lemma}

\begin{proof}
Operator $\ops$. Suppose $\classX = \ops \classX$ and $\classY = \ops\classY$, and let $A\in \mathfrak{XY}$.  Then $A$ has an ideal $I\in \classX$ with $A/I\in \classY$. Let $V\leq A$. We must show that $V\in \mathfrak{XY}$. To this end, consider the quotient algebra $(V+I)/I$ and notice that $(V+I)/I\leq A/I\in \classY = \ops\classY$. Hence $(V+I)/I\in \classY$. Using the second  isomorphism theorem we have that $(V+I)/I\cong V/(V\cap I)\in \classY$. Furthermore, $V\cap I\leq I \in \classX = \ops \classX$ and from here it follows that $V\in \mathfrak{XY}$.

Operator $\opq$. Suppose $\classX = \opq \classX$ and $\classY = \opq\classY$, and let $A\in \mathfrak{XY}$. Then $A$ has an ideal $I\in \classX$ with $A/I\in \classY$.  Let $K$ be an ideal of $A$. We show that $A/K\in \mathfrak{XY}$. To do that, consider the ideal $(I+K)/K$ of $A/K$. Certainly we have $(I+K)/K \cong I/(I\cap K) \in \opq \classX = \classX$. Additionally, 
\[  \frac{A/K}{(I+K)/K} \cong \frac{A}{I+K} \cong \frac{A/I}{(I+K)/I} \in  \opq(A/I) \subseteq \opq\classY =\classY.\]
Consequently, $A/K\in \mathfrak{XY}$.

Operator $\ops_n$. Suppose that  $\classX = \ops_n \classX$ and $\classY = \ops_n\classY$. Now we prove that $\ops_n(\classX\classY) \subseteq \classX\classY$. To this end, we show that if $K$ is a subideal of $A\in \classX\classY$, then $K\in \mathfrak{XY}$. Let $I$ be an ideal of $A$ with $I\in \classX$ and $A/I\in \classY$.  
Since $K\cap I$ is a subideal of $I\in \classX$ we have $K\cap I\in \ops_n\classX = \classX$; also $K/(K\cap I) \cong (K+I)/I \unlhd A/I\in \classY$. Therefore $K/(K\cap I) \in \ops_n\classY = \classY$.  Hence 
$K\in \mathfrak{XY}$. 
\end{proof}

\begin{lemma}
Let $\mathfrak{X}$, $\mathfrak{Y}$ and  $\mathfrak{Z}$ are classes of algebras. Then each of the following two conditions is sufficient to guarantee the associativity of the product of classes.
\begin{enumerate}
\item  $\mathfrak{X} = \opn_0 \mathfrak{X}$ and  $\mathfrak{Y} = \opq \mathfrak{Y}$.
\item $\mathfrak{X} = \bar{\ops}_n\mathfrak{X}$ and  $\mathfrak{Y} = \opq \mathfrak{Y}$
\end{enumerate}
\end{lemma}

\begin{proof}
It is obvious that $\mathfrak{X} (\mathfrak{Y}\mathfrak{Z}) \subseteq (\mathfrak{X}\mathfrak{Y}) \mathfrak{Z}$. Then  it is enough to prove that $(\mathfrak{X} \mathfrak{Y})\mathfrak{Z} \subseteq \mathfrak{X} (\mathfrak{Y} \mathfrak{Z})$. 
\begin{enumerate}
\item Let $A$ be an algebra belonging to the class $(\mathfrak{X} \mathfrak{Y})\mathfrak{Z}$. Then $A$ has a chain of subalgebras
\begin{equation}\label{cadena1}
0\leq K \unlhd I \unlhd A,
\end{equation}
such that $K\in \mathfrak{X}$, $I/K\in \mathfrak{Y}$ and $A/I\in \mathfrak{Z}$. Let $K^A$ be the ideal closure of $K$ in $A$. Then, by using \eqref{idealclosure}, we have $K^A = \sum_{k\in K} (k)$ and it follows that $K^A$ is an ideal of $A$, which belongs to $\opn_0 \mathfrak{X} = \mathfrak{X}$. Furthermore, the quotient algebra $A/K^A$  has an ideal $I/K^A \cong (I/K)/(K^A/K) \in \opq (I/K) \subseteq \opq\mathfrak{Y} = \mathfrak{Y}$, and $(A/A^K)/(I/K^A) \cong A/I\in \mathfrak{Z}$. Thus $A/K^A\in \mathfrak{YZ}$, and we have proved that $A\in \mathfrak{X} (\mathfrak{Y} \mathfrak{Z})$, as desired.

\item Let us suppose again that $A$ has a chain \eqref{cadena1} of subalgebras with the stated properties, and let $K^A$ be the ideal closure of $K$ in $A$. Since $K\in \mathfrak{X}$ and $K \unlhd\unlhd K^A$ we have that $K^A\in \bar{\ops}_n\mathfrak{X} = \mathfrak{X}$. The rest is similar to the argument in (1).
\end{enumerate}
\end{proof}

\section{A partial order on closure operations}

\begin{definition}
Let $\opec_1$ and $\opec_2$ be class maps, in particular, closure operations. We say that $\opec_1$ is contained in $\opec_2$, and write $\opec_1 \leq \opec_2$ if $\opec_1\mathfrak{X} \subseteq \opec_2 \mathfrak{X}$, 	for every algebras class $\mathfrak{X}$.
\end{definition}

\begin{remarks}
\begin{enumerate}
\item It is straightforward to verify that $\leq$ \ is a relation of partial order on the set of class maps and hence on the set of closure operations.
\item If $\opec_1$ and $\opec_2$  are closure operations, $\opec_1 \leq \opec_2$ if and only if $\opec_2$-closure invariably implies $\opec_1$-closure.
\item From the definition, it is obvious that $\opd_0 \leq \opr_0$, $\ops_n \leq \ops$ and $\opd_0 \leq \opn_0$.
\end{enumerate}
\end{remarks}

\begin{lemma}\label{lemma1.12}
Let $\opec_1$ and $\opec_2$ be closure operations. Then $\opec_1 \leq \opec_2$ if and only if every $\opec_2$-closed class is $\opec_1$-closed. 
\end{lemma}

\begin{proof}
Let us suppose that $\opec_1 \leq \opec_2$, and $\mathfrak{X} = \opec_2 \mathfrak{X}$. Then $\opec_1 \mathfrak{X} \subseteq \opec_2\mathfrak{X} = \mathfrak{X}$. Owing to the fact that $\opec_1$ is expanding, we have $\mathfrak{X}  \subseteq \opec_1\mathfrak{X}$, and therefore $\mathfrak{X} = \opec_1\mathfrak{X}$.

Reciprocally, assume that  every $\opec_2$-closed class is also $\opec_1$-closed, and let $\mathfrak{X}$ be any class of algebras. Since $\opec_2$ is a closure operation, we have 
$\mathfrak{X}  \subseteq \opec_2\mathfrak{X}$ and $\opec_2 \mathfrak{X}$ is $\opec_2$-closed. Seeing that $\opec_1$ is monotonic, it follows by assumption that $\opec_1\mathfrak{X} \subseteq \opec_1(\opec_2\mathfrak{X}) = \opec_2\mathfrak{X}$. Thus 
 $\opec_1 \leq \opec_2$.	
\end{proof}

\begin{definition}[The join of a set of closure operations]
Let $\{\opec_j\mid j\in J\}$ be a set of closure operations. We define the join $\opec := \langle \opec_j\mid j\in J\rangle$ by 
\begin{equation}\label{join}
\opec\mathfrak{X} := \bigcap \{\mathfrak{Y}\mid \mathfrak{X}\subseteq \mathfrak{Y} = \opec_j\mathfrak{Y} \ \ \text{for all} \ j\in J \}.
\end{equation}
for any algebras class $\mathfrak{X}$.
\end{definition}	

\begin{lemma}\label{lemma1.14}
Let $\{\opec_j\mid j\in J\}$ be a set of closure operations and let $\opec = \langle \opec_j\mid j\in J\rangle$ be their join. Then
\begin{enumerate}
\item the class map $\opec$ is a closure operation;
\item if $\mathfrak{X}$ is a class of algebras, then $\opec\mathfrak{X}$ is the uniquely determined, smallest  class containing $\mathfrak{X}$ which is simultaneously $\opec_j$-closed for all $j\in J$;
\item in the partial order on the set of closure operations defined above the join  $\opec$ is the least upper bound of the set $\{\opec_j\mid j\in J\}$.
\end{enumerate}
\end{lemma}

\begin{proof}
\begin{enumerate}
\item From the definition of $\opec$ follows immediately that it is expanding and monotonic. To prove the last property note that, for any class $\mathfrak{X}$, we have  
\[\{\mathfrak{Y} \mid \mathfrak{X}\subseteq \mathfrak{Y} = \opec_j\mathfrak{Y} \ \ \text{for all} \ j\in J \} = \{\mathfrak{Y} \mid \opec\mathfrak{X}\subseteq \mathfrak{Y} = \opec_j\mathfrak{Y} \ \ \text{for all} \ j\in J \},\]
and therefore $\opec(\opec\mathfrak{X}) = \opec\mathfrak{X}$. Then $\opec$ is also idempotent.

\item It follows easily from the definition that if $\opec$ is a closure operation and $\{\mathfrak{X}_j\mid j\in J\}$  is a set of $\opec$-closed classes then $\bigcap_{j\in J} \mathfrak{X}_j$ is $\opec$-closed. Thus $\opec\mathfrak{X}$ is a $\opec_j$-closed class for all $j\in J$. On the other hand, any $\opec_j$-closed class $\mathfrak{Y}$ for all $j\in J$, that  contain $\mathfrak{X}$ certainly contains $\opec\mathfrak{X}$ by definition, and therefore Statement (2) is true.
		
\item Let $\mathfrak{X}$ be a class of algebras. Since $\mathfrak{X} \subseteq \opec\mathfrak{X}$ we have for all $j\in J$
\begin{equation}
\opec_j \mathfrak{X} \subseteq \opec_j(\opec\mathfrak{X}) = \opec\mathfrak{X};
\end{equation}
that is, $\opec_j\leq \opec$ for all $j\in J$, and therefore $\opec$ is an upper bound of the set $\{\opec_j\mid j\in J\}$. Let now $\opa$ be any upper bound for the set $\{\opec_j\mid j\in J\}$ and let $\mathfrak{X}$ be any $\opa$-closed class. By Lemma \ref{lemma1.12} we have that  $\mathfrak{X}$ is a $\opec_j$-closed class for all $j\in J$. Then by (2) follows that
\begin{equation}
\opec \mathfrak{X} \subseteq \mathfrak{X} = \opa \mathfrak{X}.
\end{equation}
This shows Statement (3).
\end{enumerate} 
\end{proof}

\begin{lemma}
Let $\opec_1, \ldots, \opec_n$ be closure operations, and let   $\mathfrak{X}$ be a class of algebras. Then  $\mathfrak{X} = \opec_1\cdots \opec_n \mathfrak{X}$ if and only if $\mathfrak{X} = \langle \opec_1\cdots \opec_n\rangle\mathfrak{X}$.
\end{lemma}

\begin{proof}
Let suppose that $\mathfrak{X} = \opec_1\cdots \opec_n \mathfrak{X}$ and let $j\in \{1,\ldots, n\}$. By Lemma \ref{lemma1.14} (2) it is sufficient to prove that $\mathfrak{X} = \opec_j \mathfrak{X}$. Since $\opec_{j+1},\ldots, \opec_n$ are expanding and monotonic, we have $\mathfrak{X} \subseteq \opec_{j+1}\cdots \opec_n \mathfrak{X}$, and hence $\opec_j\mathfrak{X} \subseteq \opec_j \opec_{j+1}\cdots \opec_n \mathfrak{X}$. On the other hand, $\opec_1, \ldots, \opec_{j-1}$ are also expanding and monotonic, and then 
\begin{equation}
\opec_j\mathfrak{X} \subseteq \opec_1 \cdots \opec_{j-1} (\opec_j \mathfrak{X}) \subseteq \opec_1 \cdots \opec_{j-1} (\opec_j \opec_{j+1}\cdots \opec_n \mathfrak{X}).
\end{equation}
Hence $\opec_j\mathfrak{X} \subseteq \opec_1 \cdots \opec_n \mathfrak{X} = \mathfrak{X}$, and consequently $\mathfrak{X} = \opec_j \mathfrak{X}$ as required.

Conversely, assume that $\mathfrak{X} = \langle \opec_1\cdots \opec_n\rangle\mathfrak{X}$. Then it holds that  $\opec_j \mathfrak{X} = \mathfrak{X}$ for all $j\in \{1,\ldots, n\}$, and $\mathfrak{X} = \opec_1\cdots \opec_n \mathfrak{X}$.
\end{proof}

\begin{lemma}\label{lemma1.16}
Let $\opec_1$ and $\opec_2$ be closure operations.  Then any two of the following statements are equivalent:
\begin{enumerate}
\item The class map $\opec_1\opec_2$ is a closure operation;
\item $\opec_2\opec_1 \leq \opec_1\opec_2$
\item $\opec_1\opec_2 = \langle \opec_1,\opec_2\rangle$.
\end{enumerate}
\end{lemma}

\begin{proof}
$(1)\Rightarrow (2)$. Let $\mathfrak{X}$ be any class of algebras, and suppose that $\opec_1\opec_2$ is a closure operation. From the expanding and monotonic properties of $\opec_1$ and $\opec_2$ we have 
\begin{equation}
\opec_2 \opec_1 \mathfrak{X} \subseteq \opec_2 \opec_1  (\opec_2 \mathfrak{X}) \subseteq \opec_1 (\opec_2 \opec_1 (\opec_2 \mathfrak{X})) = (\opec_1 \opec_2)^2 \mathfrak{X} = (\opec_1 \opec_2) \mathfrak{X}.
\end{equation}
Then we have  $\opec_2\opec_1 \leq \opec_1\opec_2$.

$(2)\Rightarrow (3)$. Let $\mathfrak{X}$ be any class of algebras, and let $\mathfrak{X}\subseteq \mathfrak{Y} = \opec_1 \mathfrak{Y} = \opec_2 \mathfrak{Y}$. Then $(\opec_1 \opec_2) \mathfrak{X} \subseteq (\opec_1 \opec_2) \mathfrak{Y} = \opec_1 (\opec_2\mathfrak{Y}) = \opec_1\mathfrak{Y} = \mathfrak{Y}$. This show that $(\opec_1 \opec_2) \mathfrak{X} \subseteq \langle\opec_1, \opec_2\rangle \mathfrak{X}$. On the other hand, it is clear that $(\opec_1 \opec_2) \mathfrak{X}$ is a $\opec_1$-closed class containing $\mathfrak{X}$, and since
\[\opec_2(\opec_1 \opec_2) \mathfrak{X} = (\opec_2\opec_1) \opec_2 \mathfrak{X} = (\opec_1\opec_2) \opec_2 \mathfrak{X} = (\opec_1 \opec_2) \mathfrak{X},\]
the class $(\opec_1 \opec_2) \mathfrak{X}$ is a $\opec_2$-closed class. Therefore $\langle\opec_1, \opec_2\rangle \mathfrak{X} \subseteq (\opec_1 \opec_2) \mathfrak{X}$, and (3) holds. 

$(3)\Rightarrow (1)$. This follows from the fact that $\langle \opec_1,\opec_2\rangle$ is a closure operation.
\end{proof}

\begin{lemma}\label{lemma1.17}
The class maps products $\ops\opd_0$, $\ope_\Phi\opd_0$ and $\ope_\Phi \opq$ are closure operations.
\end{lemma}

\begin{proof}
Let $\mathfrak{X}$ be any class of algebras.
\begin{enumerate}
\item Let $A\in \opd_0\ops\mathfrak{X}$. Then there exist algebras $A_1,\ldots, A_n\in \mathfrak{X}$ with subalgebras $H_1,\ldots, H_n$, respectively such that 
\[A\cong \bigoplus\nolimits_{j=1}^n H_j.\]
Since $ \bigoplus\nolimits_{j=1}^n H_j$  can be identified with the obvious subgroup of $ \bigoplus\nolimits_{j=1}^n A_j\in \opd_0\mathfrak{X}$  we conclude that  $A\in \ops \opd_0 \mathfrak{X}$. This proves that $\opd_0\ops \leq \ops\opd_0$ and then by using Lemma \ref{lemma1.16} the map $\ops\opd_0$ is a closure operation.

\item Let $A\in \opd_0\ope_\Phi \mathfrak{X}$. Then there exist algebras $A_1,\ldots, A_n\in \ope_\Phi\mathfrak{X}$ satisfying
\[A\cong  \bigoplus\nolimits_{j=1}^n A_j.\]
For all $j\in\{1,\ldots, n\}$ we have that there exist $I_j\unlhd A_j$ with $I_j\leq \Phi(A_j)$ and $A_j/I_j\in \mathfrak{X}$. Since 
$\Phi(A) =  \bigoplus\nolimits_{j=1}^n \Phi(A_j)$, by \cite[Theorem 4.8]{Towers1}, and $A/\bigoplus_{j=1}^n I_j \cong \bigoplus_{j=1}^n (A_j/I_j)\in \opd_0 \mathfrak{X}$, we have that $A/\bigoplus_{j=1}^n I_j\in \opd_0\mathfrak{X}$. On the other hand, $I:= \bigoplus_{j=1}^n I_j \leq \bigoplus_{j=1}^n \Phi(A_j) = \Phi(A)$, and therefore $A/I\in \opd_0\mathfrak{X}$ and $A\in \ope_\Phi\opd_0 \mathfrak{X}$. Thus $\opd_0\ope_\Phi \leq \ope_\Phi\opd_0$ and we use Lemma \ref{lemma1.16} to conclude.

\item Let $A\in \opq \ope_\Phi\mathfrak{X}$. Then $A\cong H/N$,  where $N\unlhd H$ and $H$ has an ideal $I$ such that $I\leq \Phi(H)$ and $H/I \in \mathfrak{X}$. Thus
\begin{equation}\label{Berlin1}
\frac{H/N}{(I+N)/N} \cong \frac{H}{I+N} \cong \frac{H/I}{(I+N)/N} \in \opq\mathfrak{X}.
\end{equation}
Furthermore, 
\begin{equation}\label{Berlin2}
(I+N)/N\leq  (\Phi(H)+N)/N\leq \Phi(H/N).
\end{equation}
Following \eqref{Berlin1} and \eqref{Berlin2} we have $A\in \ope_\Phi \opq\mathfrak{X}$. Then $\opq \ope_\Phi \leq \ope_\Phi \opq$ and the assertion is true. 
\end{enumerate}
\end{proof}

\begin{definition}
Let $A_1, \ldots, A_r$ be algebras, and let $D:= \bigoplus_{i=1}^r A_i$ be their direct sum.
\begin{enumerate}
\item The map $\pi_i : D\longrightarrow A_i$ defined by \[\pi(a_1,\ldots,a_r) := a_i\]
is called the \texttt{projection} of $D$ onto the $i$-th component.
		
\item A subalgebra $U$ of $D$ is called \texttt{subdirect} if $\pi_i(U) = A_i$, for each $i= 1,\ldots, r$.
\end{enumerate}
\end{definition}

\begin{definition}
A subdirect product is a subalgebra $A$ of a direct product $\prod_{i\in I} A_i$ such that every induced projection (the composite $p_js: A\longrightarrow A_j$ of a projection $p_j: \prod_{i\in I} A_i\longrightarrow A_j$ with the subalgebra inclusion $s : A \longrightarrow \prod_{i\in I} A_i$) is surjective.
\end{definition}

\begin{lemma}\label{subdirect1}
Let $I_1, \ldots, I_n$ be ideals of an algebra $A$. Then the function $f : A\longrightarrow \bigoplus_{j=1}^n (A/I_j)$ defined by 
\[f(x) = (x+I_1,\ldots, x+I_n)\]
is a algebra homomorphism; its kernel is $\bigcap_{j=1}^n I_j$ and its image $f(A)$ is subdirect. In particular, $A/\bigcap\nolimits_{j=1}^n I_j$ is isomorphic with a subdirect subalgebra of $\bigoplus_{j=1}^n (A/I_j)$.
\end{lemma}	

\begin{proof}
It is clear that $f$ is a homomorphism. Moreover, $f(x) = (I_1,\ldots, I_n)$, if and only if $x\in I_j$, for all $j\in\{1,\ldots, n\}$, and therefore $\ker(f) = \bigcap_{j=1}^n I_j$. Finally,  $\pi_j(f(A)) =\{x+I_j\mid x\in A \} = A/I_j$, whence $f(A)$ is subdirect. The final assertion follows from the isomorphism theorem.
\end{proof}

\begin{lemma}\label{lemma1.18}
Let $\mathfrak{X}$   be a class of algebras.
\begin{enumerate}
\item An algebra $A$ belongs to $\opr_0\mathfrak{X}$ if and only if $A$ is isomorphic with a subdirect subalgebra of a direct sum of a finite set of $\mathfrak{X}$-algebras. 
		
\item The class maps product $\opq\opr_0$ is a closure operation.
\end{enumerate}
\end{lemma}	

\begin{proof}
\begin{enumerate}
\item Let $A\in \opr_0\mathfrak{X}$. Then there exist ideals $I_1,\ldots, I_n$ of $A$ such that $A/I_j\in \mathfrak{X}$, for all $j\in\{1,\ldots, n\}$ and $\bigcap_{j=1}^n I_j=0$. By using Lemma \ref{subdirect1} we have that $A$  is isomorphic with a subdirect subalgebra of $\bigoplus_{j=1}^n (A/I_j)$. 

Reciprocally, if $f : A\longrightarrow \bigoplus_{j=1}^n I_j$ is a monomorphism with $f(A)$ subdirect and each $I_j\in \mathfrak{X}$. We consider now the following functions:

\tikzset{node distance=2cm, auto}
\begin{center}
	\begin{tikzpicture}
\node (C) {$A$};
\node (P) [below of=C] {$\bigoplus_{j=1}^n I_j$};
\node (Ai) [right of=P] {$I_j$};
\draw[->] (C) to node {$\pi_j f$} (Ai);
\draw[->, dashed] (C) to node [swap] {$f$} (P);
\draw[->] (P) to node [swap] {$\pi_j$} (Ai);
\end{tikzpicture}
\end{center}
Since $f(A)$ is subdirect we have for all  $j\in\{1,\ldots, n\}$ that 
\begin{equation}
\mathrm{Im}(\pi_j f) = \pi_j (f(A)) = I_j.
\end{equation}
We define $N_j := \ker(\pi_j f)$. Then $A/N_j\cong I_j\in \mathfrak{X}$ and $\bigcap_{j=1}^n I_j= \ker(f)= 0$. Thus $A\in \opr_0 \mathfrak{X}$.
	
\item  Let $A\in \opr_0\opq \mathfrak{X}$. Then, using (1), there exist algebras $H_1,\ldots, H_n\in \mathfrak{X}$ and ideals $I_j\unlhd H_j$, for $j\in\{1,\ldots, n\}$, such that $A\leq \bigoplus_{j=1}^n (H_j/I_j)$ and $\pi_j(A) = H_j/I_j$, for all $j\in\{1,\ldots, n\}$.
Let $\alpha$ denote the natural homomorphism from $\bigoplus_{j=1}^n H_j$ onto $\bigoplus_{j=1}^n (H_j/I_j)$ and define $H$ as the inverse image of $A$ under $\alpha$. To prove that $H$ is subdirect in $\bigoplus_{j=1}^n H_j$ we consider the following maps:

\begin{center}
	\begin{tikzpicture}
\matrix (m) [matrix of math nodes,row sep=3em,column sep=3em,minimum width=3em]
{H\leq \bigoplus_{j=1}^n H_j & \bigoplus_{j=1}^n (H_j/I_j)\geq A \\
	H_j & H_j/I_j \\};
\path[-stealth]
(m-1-1) edge node [left] {$\pi_j$} (m-2-1)
edge [] node [above] {$\alpha$} (m-1-2)
(m-2-1.east|-m-2-2) edge node [below] {$\gamma_j$}
node [above] {} (m-2-2)
(m-1-2) edge node [right] {$\tilde{\pi}_j$} (m-2-2);
\end{tikzpicture}
\end{center}
Note that 
\[H_j/I_j\cong \tilde{\pi}_j(A) = \tilde{\pi}_j(\alpha(H)) = \gamma_j(\pi_j(H)) = \pi_j(H)/I_j.\]
Then, $\pi_j(H)=H_j$ for all $j\in\{1,\ldots, n\}$ and this shows that $H\in \opr_0\mathfrak{X}$. Moreover, $A\cong H/(H\cap \ker(\alpha)) \in \opq\opr_0\mathfrak{X}$. Thus  $\opr_0\opq \leq \opq\opr_0$, and so the map $\opq\opr_0$ is a closure operation.
\end{enumerate}
\end{proof}

\begin{remark}
If $\mathcal{S}$ is a set of algebras, we use $(\mathcal{S})$ to denote the smallest class of algebras containing $\mathcal{S}$, and when $\mathcal{S} = \{A\}$, a singleton, we write $(A)$ rather than $(\{A\})$.
\end{remark}

We end this section by describing unary closure operations and a useful exponential notation for this maps.
  
\begin{definition}
A closure operation $\opec$ is called unary if $\opec\mathfrak{X} =  \bigcup \{\opec(A) \mid A\in\mathfrak{X} \}$ for all classes $\mathfrak{X}$. If $\opec$ is a unary closure operation and $\mathfrak{X}$ a class of algebras, we define 
\[\mathfrak{X}^{\opec} = (A\mid \opec(A) \subseteq\mathfrak{X}).\] 
\end{definition}
 
\begin{remarks} 
\begin{enumerate}
\item The maps $\opq$, $\ops$ and $\ops_n$ are examples of unary operations whereas $\opd_0$, $\opr_0$ and $\opn_0$ are not.
\item It follows from definition of unary operation that $\mathfrak{X}^{\opec}$ is the unique largest $\opec$-closed class contained in $\mathfrak{X}$. In particular holds $\opec(\mathfrak{X}^{\opec}) = \mathfrak{X}^{\opec}$.
\end{enumerate}
\end{remarks}

\begin{lemma}\label{lemma1.20}
Let $\opec_1$ and $\opec_2$ be closure operations and $\mathfrak{X}$ a  $\opec_2$-closed class of algebras. If $\opec_1$  is unary and $\opec_1\opec_2 \leq \opec_2\opec_1$, then 
$\mathfrak{X}^{\opec_1}$ is a $\opec_2$-closed class.
\end{lemma}

\begin{proof}
We have $\opec_1 \opec_2 \mathfrak{X}^{\opec_1} \subseteq \opec_2 \opec_1 \mathfrak{X}^{\opec_1} = \opec_2 \mathfrak{X}^{\opec_1} \subseteq \opec_2 \mathfrak{X} = \mathfrak{X}$, and therefore $\opec_2 \mathfrak{X}^{\opec_1}$ is a $\opec_1$-closed subclass of $\mathfrak{X}$. Consequently $\opec_2 \mathfrak{X}^{\opec_1} \subseteq \mathfrak{X}^{\opec_1}$, and it follows that $\mathfrak{X}^{\opec_1}$ is $\opec_2$-closed.
\end{proof}

\section{Classes defined by closure properties}
Schunck classes, formations, varieties, and Fitting classes are example of classes of algebras which satisfy certain closure properties. This classes are the main subject of this section.

\begin{definition}
\begin{enumerate}
\item A non-empty $\opq$-closed class $\mathfrak{H}$ is called a \texttt{homomorph}. That is, $\mathfrak{H}$ contains, along with an algebra $A$, all epimorphic images of $A$.
\item A homomorph which is also $\opr_0$-closed is called a \texttt{formation}. By lemma \ref{lemma1.18} (2) a formation is precisely a $\opq\opr_0$-closed class, and classes which are simultaneously closed under $\ops$, $\opq$, and $\opd_0$ are formations. Furthermore, $\langle \opq, \opr_0 \rangle\mathfrak{X} =  \opq\opr_0\mathfrak{X}$ is the formation generated by the class $\mathfrak{X}$.
\item A homomorph which is also $\opr$-closed is called a \texttt{variety}. 
\item A homomorph which is also $\opp$-closed is called a \texttt{Schunck class}.
\item A non-empty class $\mathfrak{X}$ is called a \texttt{Fitting class} if and only if, $\langle \ops_n, \opn_0 \rangle\mathfrak{X} = \mathfrak{X}$.
\end{enumerate}	
\end{definition}

\begin{lemma}
Let $\mathfrak{X}$  be an $\opr$-closed class of algebras and $L$ an algebra. Then the set
\[\mathscr{S}(\mathfrak{X},L) := \{I\mid I\unlhd L, \ L/I\in \mathfrak{X}, \}\]
partially ordered by inclusion, has a unique minimal element, denoted by $L^\mathfrak{X}$ and called the $\mathfrak{X}$-residual of $L$. If $\mathfrak{X}$ is a formation and $\varphi : L \longrightarrow \varphi(L)$ is an epimorphism, then $\varphi(L^\mathfrak{X}) = \varphi(L)^\mathfrak{X}$.
\end{lemma}	

\begin{proof}
Let $R:= \bigcap \{I\mid I\in \mathscr{S}(\mathfrak{X},L) \}$. It is clear that $R$ is an ideal of $L$. We consider now the set
\[\mathcal{T}:= \{I/R\mid I\in \mathscr{S}(\mathfrak{X},L)\}.\] It is verified that $\bigcap \mathcal{T} = 0$ and $(L/R)/(I/R) \in \mathfrak{X}$. Therefore $L/R\in \opr \mathfrak{X} = \mathfrak{X}$ and we have $R\in \mathscr{S}(\mathfrak{X},L) $. Then $R$ is the desired smallest element of $\mathscr{S}(\mathfrak{X},L)$, which is denoted by $L^\mathfrak{X}$.
	
Let $\varphi : L \longrightarrow \varphi(L)$ be an epimorphism. Let $R:= L^\mathfrak{X}$, $T:= \varphi(L)^\mathfrak{X}$ and $I:= \varphi^{-1}(T)$, the inverse image of $T$ in $L$. Then by using the isomorphism theorem we have 
\[L/I \cong \varphi(L)/T \in  \mathfrak{X},\]
and therefore $L^\mathfrak{X}\leq I$ and so $\varphi( L^\mathfrak{X}) \leq \varphi(I) = \varphi(\varphi^{-1}(T)) = T = \varphi(L)^\mathfrak{X}$. Conversely, $\varphi(L)/ \varphi( L^\mathfrak{X}) \in \opq (L/L^\mathfrak{X}) \subseteq \opq\mathfrak{X} = \mathfrak{X}$, and then $T\leq \varphi( L^\mathfrak{X})$. Thus $\varphi(L)^\mathfrak{X} =\varphi( L^\mathfrak{X})$.
\end{proof}

\begin{definition}\label{d}
	Let $\mathfrak{X}$  be a homomorph. We say that a subalgebra $B$ of $A$ is an \texttt{$\mathfrak{X}$-projector} if \begin{enumerate}
		\item  $B \in \mathfrak{X}$, and
		\item $B\leq C\leq A$, $C_0\unlhd C$, $C/C_0\in \mathfrak{X}$ implies that $B+C_0=C$.
	\end{enumerate}
\end{definition}

We denote the set of all $\mathfrak{X}$-projectors of $A$ by $Proj_\mathfrak{X}(A)$ and write $B\in Proj_\mathfrak{X}(A)$ for $B$ is an $\mathfrak{X}$-projector of $A$. In the following we assume that all algebras $A$ are finite dimensional and solvable; that is, $A\in \mathfrak{E}\cap \mathfrak{S}$.

\begin{lemma}\label{1} If $B\in Proj_\mathfrak{X}(A)$ , then $B$ is maximal in the set $\{ C\leq A \mid C\in \mathfrak{X}\}$; in particular, if $A\in \mathfrak{X}$, then $B=A$.  
\end{lemma}
\begin{proof} If $B\leq C\leq A$ and $C\in \mathfrak{X}$, put $C_0=0$ in Definition \ref{d} (2).
\end{proof}

\begin{lemma}\label{2} \begin{enumerate}
		\item If $B\in Proj_\mathfrak{X}(A)$, and $B\leq C\leq A$, then $B\in Proj_\mathfrak{X}(C)$.
		\item If $B\in Proj_\mathfrak{X}(A)$ and $I\unlhd A$, then $B+I/I \in Proj_\mathfrak{X}(A/I)$.
	\end{enumerate}
\end{lemma}
\begin{proof} These are clear.
\end{proof}

\begin{lemma}\label{5} Suppose that $\mathfrak{X}$ is a homomorph which contains a non-zero algebra. Then the one-dimensional solvable algebra is in $\mathfrak{X}$, and $B\in Proj_\mathfrak{X}(A)$ , $A\neq 0$ implies that $B\neq 0$
\end{lemma}
\begin{proof} Suppose that $A\in Proj_\mathfrak{X}(A)$, $A\neq 0$. Since $A$ is solvable there exists $A_1\unlhd A$ such that $\dim (A/A_1)=1$. If $B\in Proj_\mathfrak{X}(A)$  and $A\neq 0$, then $B\neq 0$, by Lemma \ref{1}.
\end{proof}

\begin{lemma} Suppose that $B\in Proj_\mathfrak{X}(A)$ and $0\neq B \leq U \leq A$. If $U\unlhd V\leq A$, then $U=V$.
\end{lemma}
\begin{proof} Suppose that $U\neq V$. There is a subalgebra $W$ of $V$ such that $\dim W/U=1$, since $A$ is solvable. Then $W/U\in \mathfrak{X}$, by Lemma \ref{5}, and $B+U=W$, putting $W=C$, $U=C_0$ in Definition \ref{d} (2). But $B\leq U$, so $U=W$, a contradiction.
\end{proof}

\begin{lemma}\label{6} Suppose that $B \unlhd A$, $U/B\in Proj_\mathfrak{X}(A/B)$ and $C\in Proj_\mathfrak{X}(U)$. Then $C\in Proj_\mathfrak{X}(A)$.
\end{lemma}
\begin{proof} We use induction on $\dim A$. The result is trivial if $\dim A=0$. Suppose that $C\leq D\leq A$, $D_0\unlhd D$ and $D/D_0\in \mathfrak{X}$. Then $D+B\geq C+B=U$, by Lemma \ref{1}, and $U/B\in Proj_\mathfrak{X}(C+B/B)$, by Lemma \ref{2} (2). But $C+B/B\cong C/C\cap B$ and so $U\cap C/B\cap C \in Proj_\mathfrak{X}(C/B\cap C)$.
	\par
	
	If $D< A$, then, since $C\in  Proj_\mathfrak{X}(U\cap D)$, $C\in  Proj_\mathfrak{X}(D)$ by induction and $C+D_0=D$. Thus, $C\in  Proj_\mathfrak{X}(A)$.
	\par
	
	So suppose that $D=A$. Since $A/(D_0+B)\in \mathfrak{X}$, $U/B+(D_0+B)/B=A/B$ and so $U+B+D_0=A$. Since $U\geq B$, $U+D_0=A$. Hence $U/D_0\cap U\cong A/D_0\in \mathfrak{X}$. It follows that $C+(D_0\cap U)=U$ and $C+D_0=U+D_0=A$.
\end{proof} 

We will call an ideal $B$ of $A$ \texttt{abelian} if $B^2=0$. Note that, if $A$ has an abelian ideal, then it will have a minimal abelian ideal. However, in general, a solvable algebra may not have an abelian ideal, as we will see later.

\begin{lemma}\label{3} Let $\mathfrak{X}$ be a homomorph and let $B$ be a minimal abelian ideal of $A$. Suppose that $A/B\in \mathfrak{X}$, $A \notin \mathfrak{X}$ and $C\in  Proj_\mathfrak{X}(A)$. Then $C$ complements $B$ in $A$.
\end{lemma}
\begin{proof} We have that $A=B+C$ and $C\neq A$, from Definition \ref{d} (2). Moreover, $B\cap C$ is an abelian ideal of $A$, since $B$ is abelian. Hence $B\cap C=0$. 
\end{proof}

\begin{lemma}\label{4} Let $\mathfrak{X}$ be a formation and let $B$ be a minimal abelian ideal of $A$. Suppose that $A/B\in \mathfrak{X}$, $A\notin \mathfrak{X}$ and that $C$ complements $B$ in $A$. Then $C\in \mathfrak{X}$.
\end{lemma}
\begin{proof} We have that $C\cong A/B\in \mathfrak{X}$ and $C$ is a maximal subalgebra of $A$, so it remains to prove that, if $C_0$ is an ideal of $A$ and $A/C_0\in \mathfrak{X}$, then $A=C+C_0$. Since $\mathfrak{X}$ is a formation, $A/B\cap C_0\in \mathfrak{X}$. But $B$ is a minimal ideal of $A$, so $B\cap C_0=0$ or $B$. The former is ruled out by the fact that $A\notin \mathfrak{X}$, so $B\subseteq C_0$ and $C+C_0\supseteq C+B=A$.
\end{proof}

\begin{definition}\label{d3} If $B$ is an abelian ideal of $A$ we can consider $B$ as an $A/B$-bimodule by defining
$$ B\times A/B \rightarrow B, (b,a+B)\mapsto ba \hbox{ and } A/B\times B\rightarrow B, (a+B,b)\mapsto ab.$$ Using this we can define the split null extension, $B\rtimes A/B$, of $B$ by $A/B$ in the usual way. Then a homomorph $\mathfrak{X}$ is said to be \texttt{split} if, whenever $A\in \mathfrak{X}$ and $B$ is an abelian ideal of $A$, the split null extension of $B$ by $A/B$, also belongs to $\mathfrak{X}$.
\end{definition}

\begin{lemma} Let $\mathfrak{X}$ be a formation. Then $\mathfrak{X}$ is split.
\end{lemma}
\begin{proof} Suppose that $A\in \mathfrak{X}$ and that $B$ is an abelian ideal of $A$. Let $C$ be isomorphic to $B$ as an $A$-bimodule, and form the split null extension $X=C\rtimes A$. Then $Y=B\oplus C$ is a direct sum of two isomorphic $X$-bimodules. Let $D$ be the diagonal subalgebra of $Y$. Then $D$ is an ideal of $X$. Also
$$ A+D=A+B+C=X, A\cap D = 0 \hbox{ and } X/D \cong X/C\cong A\in \mathfrak{X}.
$$ Thus $X=X/C\cap D\in \mathfrak{X}$ and so $X/B\in \mathfrak{X}$. But $X/B$ is the split null extension of $B+C/B (\cong B)$ by $A/B$.
\end{proof}

\begin{definition}\label{d4} A homomorph $\mathfrak{X}$ is \texttt{saturated} if all solvable algebras which have an abelian ideal have $\mathfrak{X}$-projectors.
\end{definition}

\begin{theorem} Let $\mathfrak{X}\subseteq \mathfrak{S}$ be a non-empty homomorph. Then the following are equivalent
\begin{enumerate}
\item $\mathfrak{X}$ is saturated; and
\item if $A\notin \mathfrak{X}$ and $A$ has a minimal abelian ideal $B$ with $A/B\in \mathfrak{X}$, then $ Proj_\mathfrak{X}(A) \not = \emptyset$.
\end{enumerate}
\end{theorem} 
\begin{proof}
$(1) \Rightarrow (2)$ is clear.
\par

$(2) \Rightarrow (1)$: Suppose that (ii) holds. We need to show that, if $A$ is solvable and has an abelian ideal, $ Proj_\mathfrak{X}(A)\not = \emptyset$. Let $A$ be a minimal counter-example. Then there exists a subalgebra $S\subseteq A$ such that $B\subseteq S$, $S/B\in  Proj_\mathfrak{X}(A/B)$.
\par

If $S\neq A$ then there exists $C\in  Proj_\mathfrak{X}(S)$, by the minimality of $A$, and $B\in  Proj_\mathfrak{X}(A)$, by Lemma \ref{2}. If $S=A$, then $A/B\in \mathfrak{X}$. Either $A\in \mathfrak{X}$ and $A\in  Proj_\mathfrak{X}(A)$ or $A\notin \mathfrak{X}$ and $ Proj_\mathfrak{X}(A)\not = \emptyset$ by hypothesis.
\end{proof}

\begin{corollary}\label{c} Let $\mathfrak{X}$ be a non-empty formation. Then the following are equivalent:
\begin{enumerate}
\item $\mathfrak{X}$ is saturated; and
\item if $A\notin \mathfrak{X}$, $B$ is a minimal abelian ideal of $A$ and $A/B\in \mathfrak{X}$, then there is a complement to $B$ in $A$.
\end{enumerate}
\end{corollary}
\begin{proof} Suppose that $A\notin \mathfrak{X}$, $B$ is a minimal abelian ideal of $A$ and $A/B\in \mathfrak{X}$. By Lemmas \ref{3} and \ref{4}, $C\in  Proj_\mathfrak{X}(A)$ if and only if $C$ complements $B$ in $A$.
\end{proof}

\begin{lemma} Suppose that $\mathfrak{X}$ is a non-empty formation and $A/\Phi(A)\in \mathfrak{X}$ implies that $A\in \mathfrak{X}$. Then $\mathfrak{X}$ is saturated.
\end{lemma}
\begin{proof} We must show that, if $A\notin \mathfrak{X}$, $B$ is a minimal abelian ideal of $A$ and $A/B\in \mathfrak{X}$, then there is a complement $C$ to $B$ in $A$, by Corollary \ref{c}. If $B\subseteq \Phi(A)$ then $A/\Phi(A)\in \mathfrak{X}$, so $A\in \mathfrak{X}$, a contradiction. Hence $B\not \subseteq \Phi(A)$, and so there is a maximal subalgebra $C$ of $A$ with $B\not \subseteq C$.Then $A=B+C$ and $B\cap C=0$.
\end{proof}
 
\section{Conclusions and future work}
In the above, we have seen that the basic language and results from the theory of classes of groups and closure operations can be applied successfully in this general setting. But to develop the theories of formations, Schunck classes and Fitting theory appears to require extra restrictions on the algebras under consideration. Difficulties arise because the algebras in the derived series, $A^{(n)}$ are not necessarily ideals unless $n=2$ and minimal ideals of solvable algebras are not necessarily abelian.  As a result, the natural examples of saturated formations and Schunck classes, such as nilpotent algebras, are no longer valid unless we add extra restrictions to the definitions, as the following example shows.

\begin{example} 
Let $A$ be the three-dimensional algebra over a field $K$ spanned by $x,y,z$ with products $x^2=z$, $yz=z$ and $zx=y$, all other products being zero. Then $A^{(2)}=Ky+Kz$, $A^{(3)}=Kz$, $A^{(4)}=0$, so this is a solvable algebra. However, $A^{(2)}$ is a minimal ideal of $A$ and is not abelian; indeed, it is not even nilpotent. The only proper subalgebras of $A$ are $Ky$, $Kz$ and $A^{(2)}$, and the only nilpotent ones are one-dimensional.

Then $Ky$ and $Kz$ are the only possible $\mathfrak{N}$-projectors. But both are subalgebras of $A$, $A^{(2)}$ is an ideal of $A$ and $A/A^{(2)}\in \mathfrak{N}$, and $Ky+A^{(2)}=Kz+A^{(2)}=A^{(2)}\not = A$. Hence $A$ has no $\mathfrak{N}$-projectors. Clearly $\mathfrak{N}$ is a formation, but would not saturated if we required that all solvable algebras have $\mathfrak{N}$-projectors. Notice also that $\phi(A)=A^{(2)}$ is not nilpotent.
\end{example}

So, the next step is to find a class of non-associative algebras for which the theory can be progressed beyond this point. We would, of course, wish such a class to include the important varieties for which an advanced structure theory exists and for which such a theory has already been developed, such as Lie, Leibniz and Malcev algebras.

\bibliographystyle{amsplain}

\end{document}